\documentclass[11pt,a4paper]{amsart}
\usepackage{bm}
\usepackage{amsmath}
\usepackage{amsfonts}
\usepackage{amssymb}
\usepackage{amsthm}
\usepackage{graphicx}
\usepackage{color}
\usepackage{url}
\usepackage{a4wide}
\usepackage[numbers,square,sort&compress]{natbib}
\theoremstyle{plain}

\newtheorem{theorem}{Theorem}[section]
\newtheorem{corollary}[theorem]{Corollary}
\newtheorem{definition}[theorem]{Definition}

\newtheorem{lemma}[theorem]{Lemma}

\newtheorem{proposition}[theorem]{Proposition}
\newtheorem{remark}[theorem]{Remark}

\numberwithin{equation}{section}
\newcommand{\tr}{\mathop{\mathrm{Tr}}}

\newcommand{\bbr}{\mathbb{R}}

\begin{document}

\title{On strong solutions for positive definite jump diffusions}
\author[E.~Mayerhofer]{Eberhard Mayerhofer}
\address{Vienna Institute of Finance, University of Vienna and Vienna University of Economics and Business Administration, Heiligenst\"adterstrasse 46-48, 1190 Vienna, Austria}
\email{eberhard.mayerhofer@vif.ac.at}
\author[O. Pfaffel]{Oliver Pfaffel}
\address{TUM Institute for Advanced Study \& Zentrum Mathematik, Technische Universit\"at M\"unchen, Boltzmannstra\ss e 3, D-85747 Garching bei M\"unchen, Germany}
\email{pfaffel@ma.tum.de}
\author[R. Stelzer]{Robert Stelzer}
\address{Institute of Mathematical Finance, Ulm University, Helmholtzstra\ss e 18, D-89081 Ulm, Germany}
\email{robert.stelzer@uni-ulm.de}

\thanks{The authors thank the anonymous referees for most helpful comments on a previous version of the present paper, in particular, for the suggestion of a shorter proof for Proposition \ref{lem: MCKean}.\\
E.M. gratefully acknowledges financial support from WWTF (Vienna Science and Technology Fund), O.P. and R.S. the one from Technische Universit\"at M\"unchen - Institute for Advanced Study funded by the German Excellence Initiative and O.P. additionally the one from the International Graduate School of Science and Engineering (IGSEE)}
\date{}
\begin{abstract}
We show the existence of unique global strong solutions of a class of stochastic differential equations on
the cone of symmetric positive definite matrices. Our result includes affine diffusion processes and therefore extends considerably the known statements concerning Wishart processes, which have recently been extensively employed in financial mathematics.

Moreover, we  consider   stochastic differential equations where the diffusion coefficient is given by the $\alpha$-th positive semidefinite power of the process itself with $0.5<\alpha<1$ and obtain existence conditions for them. In the case of a diffusion coefficient which is linear in the process we likewise get a positive definite analogue of the univariate GARCH diffusions.
\end{abstract}
\maketitle

\section{Introduction}
A result of the general theory for affine Markov processes on
the cone $S_d^+$ of symmetric positive semidefinite matrices developed in \cite{CFMT} is that for a $d\times d$ matrix-valued standard Brownian motion $B$,  $d\times d$ matrices $Q$ and $\beta$, a symmetric constant drift $b$,  and a positive linear drift $\Gamma: S_d^+\rightarrow S_d^+$, {\it weak global solutions} exist to the stochastic differential equation (SDE)
\begin{align}\label{eq: affine process}
dX_t&=\sqrt{X_t}dB_tQ+Q^\top dB_t^\top\sqrt{X_t}+ (X_t \beta+\beta^\top
X_t+\Gamma(X_t)+b)dt,\\
X_0&=x\in S_d^{+},\nonumber
\end{align}
whenever $b- (d-1)Q^\top Q\in S_d^+$. Above $\sqrt{X}$ denotes the unique positive semidefinite square root of a matrix $X\in{S}_d^+$. For $\Gamma=0$ solutions to the SDE \eqref{eq: affine process} are called Wishart processes and their existence has been considered in detail in  the
fundamental paper by Marie-France Bru \cite{bru}. Further probabilistic investigations on properties of Wishart processes have been carried out  in
 \cite{donatimartin2008,donatimartin,GraczykVostrikova2007}, for instance, and references therein.

In the present paper, we focus on the existence of global {\it strong solutions} of \eqref{eq: affine process} and generalisations of it including jumps and more general diffusion coefficients. Because of the non-Lipschitz diffusion at the boundary of the cone, this problem is a quite delicate one -- a-priori
it is only clear that a unique local solution of \eqref{eq: affine process} exists until $X_t$ hits the boundary of $S_d^+$, since the SDE is locally Lipschitz in the interior of $S_d^+$. Furthermore, known results for pathwise uniqueness, for instance, that of the seminal paper of Yamada and Watanabe \cite[Corollary 3]{Yamada}, are essentially one-dimensional, and therefore do not apply. Hence, the present setting seems to be more complicated than, for instance,  the canonical affine  one (concerning diffusions on $\mathbb R_+^m\times\mathbb R^n$, \cite[Lemma 8.2]{ADPTA}).

Positive semidefinite matrix valued processes are increasingly used in finance, particularly for stochastic modelling
of multivariate  stochastic volatility phenomena in equity and fixed income models, see \cite{bur_cie_tro_07, buraschiporchiatrojani, fonsecagrasselliielpo1, fonsecagrasselliielpo2, fonsecaetal2,fonsecaetal1, gourierouxsufana, gou_suf_04,  grasselli, pigorschstelzer2}. See also \cite{CFMT} and the references therein. Most papers mentioned  use Bru's class of Wishart diffusions, as this results in multivariate analogues of the popular Heston stochastic volatility model and its extensions, Ornstein-Uhlenbeck type processes (\cite{ pigorschstelzer2}) giving a multivariate generalisation of the popular model of \cite{Barndorffetal2001c} or a combination of both (\cite{LeippoldTrojani2008}). This motivated  the research of \cite{CFMT} on positive semidefinite affine processes  including all the aforementioned models and generalising the results of \cite{dfs}, which covered all of these models in the univariate setting. Appropriate multivariate models are especially important for issues like portfolio optimisation, portfolio risk management and the pricing of options depending on several underlyings, which are
 heavily influenced by the dependence structure.

Clearly $S_d^+$-valued processes  model the  covariances, not the correlations, which are, however, preferable when interpreting the dependence structure. The results of the present paper are particularly  relevant, when one wants to derive correlation dynamics (see e.g., \cite{bur_cie_tro_07, buraschiporchiatrojani}), because one needs to assume boundary non-attainment conditions for a rigorous derivation.

The name ``Wishart process'' is, unfortunately, not always used in the same way in the literature. We follow the above cited applied papers in finance and call any solution to \eqref{eq: affine process} with $\Gamma=0$ ``Wishart process'' whereas in most of the previous probabilistic literature ``Wishart process'' also means $\beta=0$ and the ``Wishart processes with drift'' of \cite{donatimartin} are not even special cases of our ``Wishart processes''. For $\Gamma=\beta=0$ and $b=nQ^TQ$ with $n\in\mathbb{N}$ one may also speak of a ``squared Ornstein-Uhlenbeck process''. In the univariate case the name ``Wishart process'' is not used, instead one typically uses ``Cox-Ingersoll-Ross process'' in the financial and ``squared Bessel process'' in the  probability literature.

However, in this paper we do not limit ourselves to the analysis of \eqref{eq: affine process}. Instead, as a special case of a considerably more general result, we consider a similar SDE allowing for a general (not necessarily linear) drift $\Gamma$ and an additional jump part of finite variation. This implies that many L\'evy-driven SDEs on $S_d^+$ like the positive semidefinite Ornstein-Uhlenbeck (OU) type processes (see \cite{barndorffstelzer,PigorschetStelzer2007b}) or the volatility process of a multivariate COGARCH process (see \cite{Stelzer2008mcg}), where the existence of global strong solutions has previously been shown by path-wise arguments, are special cases of our setting. Thus our results allow to consider certain ``jump diffusions'' (in the sense of \cite{Contetal2004}) , viz. mixtures of such jump processes and Wishart diffusions, in applications.

It should be noted that \cite{bru} also contains results on strong solutions for Wishart processes (see
our upcoming
 Proposition \ref{th: bru} and Remark \ref{remark on bru}),
however, they are derived under strong parametric restrictions, because her method
requires an application of Girsanov's theorem. The latter is based on a martingale criterion, which in the matrix valued setting seems hard to verify. Also, the general result (with a non-vanishing linear drift) only holds until the first time when two of the eigenvalues of the process collide. Our approach generalises her method of proof for the case $\beta=0$ (vanishing linear drift) and avoids change of measure techniques.

The most general result of our paper, Theorem \ref{th: strong solutions general}, also opens the way to use positive semidefinite extensions of the univariate GARCH diffusions of \cite{Nelson1990} or of so-called generalised Cox-Ingersoll-Ross models (cf. e.g. \cite{BorkovecKlueppelberg1998,FasenKlueppelbergLindner2006}), where the square root in the diffusion part of \eqref{eq: affine process} is replaced by the $\alpha$-th positive semidefinite power with $\alpha\in[1/2,1]$, in applications (see Corollary \ref{th:gcir}).

The remainder of the paper is structured as follows. In the subsequent section we summarise some notation and preliminaries. In Section \ref{sec:results} we state our main result, Theorem \ref{th: strong solutions general}, and its corollaries applying to Wishart processes, matrix-variate generalised Cox-Ingersoll-Ross and GARCH diffusions. Moreover, we compare our results to the work of Bru which is recalled in Proposition \ref{th: bru}. In the following section we gradually develop the proof of our result.
Our method relies on a generalisation of the so-called {\it McKean's argument}, but avoids the use of Girsanov's theorem. In Section \ref{section: prelim} we thus provide a self-contained proof of a generalisation of {\it McKean's argument} and then deliver the proof of Theorem \ref{th: strong solutions general}  in Section \ref{sec: proof}. We conclude the paper with some final remarks in Section \ref{sec: final}.

\section{Notation and general set-up}
We assume given an appropriate filtered probability space $(\Omega,\mathcal{F},\mathbb{P},\allowbreak(\mathcal{F}_t)_{t\in\mathbb R_+})$ satisfying the usual hypotheses (complete and right-continuous filtration) and rich enough to support all processes occurring. For short, we sometimes write just $\Omega$ when actually referring to this filtered probability space. $B$ is a $d\times d$ standard Brownian motion on $\Omega$ and
 $d\in \mathbb{N}$ always denotes the dimension. Furthermore, we use the following notation, definitions and setting:
\begin{itemize}
\item $\mathbb R_+:=[0,\infty)$, $M_d$ is the set of real valued $d\times d$
matrices and $I_d$ is the identity matrix.
\item $S_d\subset M_d$ is the space of symmetric matrices, and
$S_d^+\subset S_d$ is the cone of symmetric positive semidefinite
matrices in $S_d$ and  $S_d^{++}$ its interior, i.e. the positive definite matrices. The partial order on $S_d$ induced by the cone is
denoted by $\preceq$, and $x\succ 0$, if and only if $x\in S_d^{++}$. We endow $S_d$ with the scalar product $\langle x,y\rangle:=\tr(xy)$, where $\tr(A)$
denotes the trace of $A\in M_d$. $\|\,\cdot\,\|$ denotes the associated norm, and $d(x,\partial S_d^+)=\inf_{y\in\partial S_d^+}\|x-y\|$ is the distance of $x\in S_d^+$ to the boundary $\partial S_d^+$.
\item The usual tensor (Kronecker) product of two matrices $A,B$ is denoted by $A\otimes B$ and the vectorisation operator mapping $M_d$ to $\mathbb{R}^{d^2}$ by stacking the columns of a matrix $A$ below each other is denoted by $\operatorname{vec}(A)$ (see \cite[Chapter 4]{horn} for more details).
\item A function $f:S_d^{++}\to U$ with $U$ being (a subset of) a normed space is called \emph{locally Lipschitz} if $\|f(x)-f(y)\|\leq K(C)\|x-y\| \,\forall\, x,y\in C$ for all compacts $C\subset U$. $f$ is said to have \emph{linear growth} if $\|f(x)\|^2\leq K(1+\|x\|^2) \,\forall\, x\in S_d^{++}$.
\item An $S_d$-valued c\`adl\`ag adapted stochastic process $X$ is called $S_d^+$-increasing, if $X_t\succeq X_s$ a.s. for all $t>s\geq 0$. Such a process is necessarily of finite variation on compacts by \cite[Lemma 5.21]{barndorffstelzer} and hence a semimartingale. We call it of \emph{pure jump type} provided $X_t=X_0+\sum_{0<s\leq t}\Delta X_s$, where $\Delta X_s=X_s-X_{s-}$.
\end{itemize}
For the necessary background on stochastic analysis we refer to one of the standard references like \cite{jacod,protter,revuzyor}. Moreover, we frequently employ stochastic integrals where the integrands or integrators are matrix- or even linear-operator valued. Thus, we briefly explain how they have to be understood.
Let $(A_t)_{t\in\bbr^+}$ in $M_{d}$, $(B_t)_{t\in\bbr^+}$ in $M_{d}$ be c\`adl\`ag and adapted processes and $ (L_t)_{t\in\bbr^+}$ in $M_{d}$   be a semimartingale (i.e. each element is a semimartingale). Then we denote by $\int_0^t A_{s-}dL_s B_{s-}$ the  matrix $C_t$ in $M_{d}$ which has   ${ij}$-th element $C_{ij,t}=\sum_{k=1}^d\sum_{l=1}^d\int_0^t A_{ik,s-}B_{lj,s-}dL_{kl,s}$. Equivalently such an integral can be understood in the sense of \cite{Metivieretal1980} by identifying it with the integral $\int_0^t \mathbf{A}_{s-}dL_s$ with $\mathbf{A}_t$ being for each fixed $t$ the linear operator $M_{d}\to M_{d},\: X\mapsto A_t X B_t$ and $L$ being a semimartingale in the Hilbert space $ M_{d}$.
 Stochastic integrals of the form $\int_0^t K(X_{s-})dJ_s$ with $J$ being a semimartingale in $M_{d}$ (coordinatewise or equivalently as in \cite[Section 10]{Metivieretal1980} where the equivalence easily follows from \cite[Section 10.9]{Metivieretal1980} and by noting that on a finite dimensional Hilbert space all norms are equivalent) and $K(x):M_{d}\to M_{d}$ a linear operator for all $x$ can be understood again as in \cite{Metivieretal1980}. Alternatively, one can equivalently identify $M_{d}$ with  $\bbr^{d^2}$ using the $\mathrm{vec}$-operator and $K(x)$ with a matrix in $M_{d^2,d^2}$ and then define the stochastic integral coordinatewise as above.
\section{Statement of the main results}\label{sec:results}
\subsection{Wishart diffusions with jumps}~\\
In order to illustrate the context of our result and, because it is of most relevance in applications, we discuss first the special case of Wishart diffusions with jumps.  For $Q\in M_d$, $\delta> d-1$, $\beta\in M_d$ and an $M_d$-valued standard Brownian motion $B$, a Wishart
process is the strong   solution of the equation
\begin{align}\label{eq: bru}
dX_t&=\sqrt{X_t}dB_tQ+Q^\top dB_t^\top\sqrt{X_t}+(X_t \beta+\beta^\top
X_t+\delta Q^\top Q)dt,\\
 X_0&=x\in S_d^{++},\nonumber
\end{align}
on the maximal stochastic interval $[0, T_x)$, where $T_x$ is naturally defined
as
\[
T_x=\inf\{t>0\,:\, X_t\in\partial S_d^+\}.
\]
That such a unique local strong solution, which does not explode before or at time $T_x$, exists, follows from standard SDE theory, since all the coefficients in \eqref{eq: bru} are locally Lipschitz and of linear growth on $S_d^{++}$. To be more precise, this follows by appropriately localising the usual results as e.g. in \cite[Chapter V]{protter} or by variations of the proofs in \cite[Chapter 3]{Metivieretal1980}. A localisation procedure adapted particularly to certain convex sets like $S_d^+$ is  presented in detail in \cite[Section 6.7]{Stelzer2007th}.

The following is a summary of the results \cite[Theorem 2, 2' and 2'']{bru} -- the to the best of our knowledge only known results regarding strong existence of Wishart processes:
\begin{proposition}\label{th: bru}
Let $\delta\geq d+1$.
\begin{enumerate}
\item If $Q=I_d$ and $\beta=0$, then $T_x=\infty$.
\end{enumerate}
Suppose additionally that the $d$ eigenvalues of $x$ are distinct.
\begin{enumerate}\setcounter{enumi}{1}
\item\label{stat: bru1}  If $Q\in S_d^{++}$, $-\beta\in S_d^+$ such that $\beta$ and $Q$ commute,
then there exists a solution $(X_t)_{t\in\mathbb{R}_+}$  of \eqref{eq: bru} until the first time $\tau_x$
when two of the eigenvalues  of $X_t$ collide.
\item \label{stat: bru2} If $\beta=\beta_0 I_d$ and $Q=\gamma I_d$, where $\beta_0,\gamma\in\mathbb R$,
then $T_x=\infty$ for the solution of $(X_t)_{t\in\mathbb{R}_+}$ of \eqref{eq: bru}.
\end{enumerate}
Consequently, for the respective choice of parameters, there exist unique global strong $S_d^{++}$-valued solutions of the SDE \eqref{eq: bru}  on $[0,\tau_x)$ resp.\ on all of $[0,\infty)$.
\end{proposition}

The upcoming general Theorem \ref{th: strong solutions general} implies  the following result for a generalisation of the Wishart SDE allowing for additional jumps and a non-linear drift $\Gamma$.
\begin{corollary}\label{th: strong solutions}
Let $b\in S_d$, $Q\in M_d$, $\beta\in M_d$, and let
\begin{itemize}
\item  $J$ be an $S_d$-valued c\`adl\`ag adapted process which is $S_d^+$-increasing and of pure jump type,
\item $\Gamma:S_d^{++}\to S_d^+$ be a locally Lipschitz function of linear growth and
\item $K: S_d^{++}\to L(S_d^+,S_d^+)$ (the linear operators on $S_d$ mapping $S_d^+$ into $S_d^+$) be a locally Lipschitz function of linear growth.
\end{itemize}
If $b\succeq(d+1)Q^{\top}Q$, then the SDE
\begin{align}\label{gensde}
dX_t=&\,\sqrt{X_{t-}}dB_tQ+Q^\top dB_t^{\top}\sqrt{X_{t-}}+ (X_{t-} \beta+\beta^\top
X_{t-}+\Gamma(X_{t-})+b)dt+K(X_{t-})dJ_t,\\
X_0=&\,x\in S_d^{++},\nonumber
\end{align}
has a unique adapted c\`adl\`ag global strong solution $(X_t)_{t\in\mathbb{R}_+}$ on $S_d^{++}$.
 In particular we have $T_x:=\inf\{t\geq 0\,:\, X_{t-}\in\partial S_d^+\mbox { or }X_t\not\in S_d^{++}\}=\inf\{t\geq 0\,:\, X_{t-}\in\partial S_d^+\}=\infty$ almost surely.
\end{corollary}
\begin{proof}
For the term on the right hand side of the upcoming condition \eqref{duplo enigmato drifto conditione} we obtain
\begin{align*}
 \tr(2\beta)+\tr(\Gamma(x)x^{-1})+\tr((b-(d+1) Q^{\top}Q)x^{-1})\geq 2\tr(\beta),
\end{align*}
 noting that $x^{-1}$, $\Gamma(x)$ and $ b-(d+1) Q^{\top}Q$ are positive semidefinite and that $S_d^+$ is a selfdual cone, which implies that $\tr(zy)\geq 0$ for any $z,y\in S_d^+$. Setting $c(t)=2\tr(\beta)$ an application of Theorem  \ref{th: strong solutions general} concludes.
\end{proof}

By choosing $\Gamma$ linear and $J=0$, we obtain a result for \eqref{eq: affine process} which considerably generalises Proposition \ref{th: bru}.
\begin{remark}
\begin{enumerate}
\item In the univariate case the condition $b\succeq(d+1)Q^{\top}Q$ is known to be also necessary for boundary non-attainment (see \cite[Chapter XI]{revuzyor}). 
\item A possible choice for $J$ 
 is a matrix subordinator without drift (see \cite{BarndorffetPerez2005}), i.e. an $S_d^+$-increasing L\'evy process. By choosing $\Gamma\neq 0$ in \eqref{gensde} appropriately our results also apply to SDEs involving matrix subordinators with a non-vanishing drift.
\item Setting $Q=0$, $\Gamma=0$, $K$ to the identity and $b$ equal to the drift of the matrix subordinator, Equation \eqref{gensde} becomes the SDE of a positive definite OU type process, \cite{barndorffstelzer,PigorschetStelzer2007b}. Likewise, it is straightforward to see that the SDE of the volatility process $Y$ of the multivariate COGARCH process of \cite{Stelzer2008mcg} is a special case of \eqref{gensde}.
\item An OU--type process on the positive semidefinite matrices is necessarily driven by a L\'evy process of finite variation having positive semidefinite jumps only (follows by slightly adapting the arguments in the proof of \cite[Theorem 4.9]{PigorschetStelzer2007b}). This entails that a generalisation of the above result to a more general jump behaviour requires additional technical restrictions.
\end{enumerate}
\end{remark}\vspace*{0.2cm}

\subsection{The general SDE and  existence result}~\\
The main result of this paper is the  following general
theorem concerning non-attainment of the boundary of $S_d^{+}$ and the existence of a unique global
strong solution for a generalisation of the SDE \eqref{eq: affine process}. The proof of this result is gradually developed in the next sections.

\begin{theorem}\label{th: strong solutions general}
Let
\begin{itemize}
\item $F, G: \mathbb R_+\times S_d^{++}\rightarrow
M_d$, be functions such that $G^\top\otimes F$ given by $G^\top\otimes F  (t,x)= (G(t,x))^\top\otimes F(t,x)$ is  locally Lipschitz and of linear growth,
\item $H: \mathbb R_+\times S_d^{++}\rightarrow
S_d$ be locally Lipschitz and of linear growth,
\item  $J$ be an $S_d$-valued c\`adl\`ag adapted process which is $S_d^+$-increasing and of pure jump type,
\item and $K: S_d^{++}\to L(S_d^+,S_d^+)$ (the linear operators on $S_d$ mapping $S_d^+$ into $S_d^+$) be a locally Lipschitz function of linear growth.
\end{itemize}
Suppose that there exists a  function $c:\mathbb{R}_+\to\mathbb{R}$ which is locally integrable, i.e. $\int_0^s|c(t)|dt<\infty$ for all $s\in\mathbb{R}^+$, such that
\begin{equation}\label{duplo enigmato drifto conditione}
c(t)\leq\tr(H(t,x)x^{-1})-\tr(f(t,x)x^{-1})\tr(g(t,x)x^{-1})-\tr(f(t,x)x^{-1}g(t,x)x^{-1})
\end{equation}
for all $x\in S_d^{++}$ and $t\in\mathbb{R}_+$ where  $f(t,x):=F(t,x)F(t,x)^\top,$ $ g(t,x)=G(t,x)^\top G(t,x)$.

Then the SDE
\begin{align}\label{gensde1}
dX_t=&F(t,X_{t-})dB_tG(t,X_{t-})+G(t,X_{t-})^\top dB_t^\top F(t,X_{t-})^\top\\&+
H(t,X_{t-})dt+K(X_{t-})dJ_t,\nonumber\\
X_0=&x\in S_d^{++},\nonumber
\end{align}
has a unique adapted c\`adl\`ag global strong solution
$(X_t)_{t\in\mathbb{R}_+}$ on $S_d^{++}$.

In particular, we have $T_x:=\inf\{t\geq 0\,:\, X_{t-}\in\partial S_d^+\mbox { or }X_t\not\in S_d^{++}\}=\inf\{t\geq 0\,:\, X_{t-}\in\partial S_d^+\}=\infty$ almost surely.
\end{theorem}
\vspace*{0.2cm}

\subsection{Positive definite extensions of generalised Cox-Ingersoll-Ross processes and GARCH diffusions}~\\
In the univariate case generalised Cox-Ingersoll-Ross (GCIR) processes given by the SDE $dx_t=(b+ax_t)dt+qx_t^{\alpha}dB_t$ with $b\geq0, q>0, a\in\mathbb{R}$ and $\alpha\in [1/2,1]$ are -- as discussed in the introduction -- of relevance in financial modelling. $\alpha=1/2$ corresponds, of course, to the already discussed Bessel  case, whereas $\alpha=1$ gives the so-called GARCH diffusions. Given the popularity of the Wishart based models in nowadays finance, it seems natural to consider also positive semidefinite extensions of the GCIR processes. An application of our general theorem to the  case where
$F(X)=X^{\alpha}$, $G(X)=Q$ with $\alpha\in[1/2,1]$ yields:
\begin{corollary}\label{th:gcir}
(i) Let $\alpha\in[1/2,1]$, $b\in S_d$, $Q\in M_d$, $\beta\in M_d$, and let
\begin{itemize}
\item  $J$ be an $S_d$-valued c\`adl\`ag adapted process which is $S_d^+$-increasing and of pure jump type,
\item $\Gamma:S_d^{++}\to S_d^+$ be a locally Lipschitz function of linear growth and
\item $K: S_d^{++}\to L(S_d^+,S_d^+)$ (the linear operators on $S_d$ mapping $S_d^+$ into $S_d^+$) be a locally Lipschitz function of linear growth.
\end{itemize}
 Suppose that for all $x\in S_d^{++}$
\begin{equation}\label{eq:condalpha}
 \tr(\Gamma(x)x^{-1}+bx^{-1})\geq \tr(x^{2\alpha-1})\tr(Q^\top Qx^{-1})+ \tr(x^{2\alpha-2}Q^\top Q).
\end{equation}

 Then the SDE
\begin{align}\label{gensde alpha}
dX_t&=X_{t-}^\alpha B_tQ+Q^\top dB_t^\top X_{t-}^\alpha+ (X_{t-}
\beta+\beta^\top
X_{t-}+\Gamma(X_{t-})+b)dt+K(X_{t-})dJ_t,\\
X_0&=x\in S_d^{++},\nonumber
\end{align}
has a unique adapted c\`adl\`ag global strong solution
$(X_t)_{t\in\mathbb{R}_+}$ on $S_d^{++}$. In particular we have
$T_x:=\inf\{t\geq 0\,:\, X_{t-}\in\partial S_d^+\mbox { or
}X_t\not\in S_d^{++}\}=\inf\{t\geq 0\,:\, X_{t-}\in\partial
S_d^+\}=\infty$ almost surely.

(ii) Any of the following sets of conditions implies \eqref{eq:condalpha}:
\begin{enumerate}
 \item[(a)] $
b+\Gamma(x)\succeq \tr(x^{2\alpha-1})Q^\top Q
+x^{\alpha-1/2}Q^\top Qx^{\alpha-1/2}$ for all $x\in S_d^{++}$.
\item[(b)] $
b+\Gamma(x)\succeq \tr(x^{2\alpha-1})Q^\top Q
+\lambda_{Q^\top Q}x^{2\alpha-1}$ for all $x\in S_d^{++}$ with $\lambda_{Q^\top Q}$ denoting the largest eigenvalue of $Q^\top Q$.
\item[(c)] $\alpha=1$ and $b+\Gamma(x)\succeq \tr(x) Q^\top Q+\lambda_{Q^\top Q}x$ for all $x\in S_d^{++}$ .
\item[(d)] $b\succeq 0$ and $\Gamma(x)\succeq 2\tr(x^{2\alpha-1})Q^\top Q$ for all $x\in S_d^{++}$.
\item[(e)] $b\succeq 0$ and $\Gamma(x)\succeq 2\left(\tr(x)+d(2\alpha-1)^{2-2\alpha}\right)Q^\top Q$  for all $x\in S_d^{++}$ (and setting $0^0:=1$ for $\alpha=1/2$).
\item[(f)] $b\succeq 0$ and $\Gamma(x)\succeq 2(\tr(x)+d)Q^\top Q$  for all $x\in S_d^{++}$.
\item[(g)] $\alpha>1/2$, $d=1$, $\Gamma(x)\geq 0$  for all $x\in \mathbb{R}_+$ and $b>0$.
\end{enumerate}
\end{corollary}
\begin{proof}
 One easily calculates the right hand side of \eqref{duplo enigmato drifto conditione} to be equal to
$\tr(2\beta+\Gamma(x)x^{-1}+bx^{-1})- \tr(x^{2\alpha-1})\tr(Q^\top Qx^{-1})- \tr(x^{2\alpha-2}Q^\top Q)$ and hence (i) follows from Theorem \ref{th: strong solutions general}.

Turning to the proof of (ii) using the selfduality of $S_d^+$ as in the proof of Corollary \ref{th: strong solutions} gives (a). Next we observe that $Q^\top Q\preceq\lambda_{Q^\top Q} I_d$  and, hence, $x^{\alpha-1/2}Q^\top Qx^{\alpha-1/2}\preceq  \lambda_{Q^\top Q}x^{2\alpha-1}$. This gives (b) and (c) is simply the special case for $\alpha=1$.

Since for $A,B\in S_d^+$ we have that $\tr(AB)\leq \tr(A)\tr(B)$ due to the Cauchy-Schwarz inequality and the elementary inequality $\sqrt{a+b}\leq \sqrt{a}+\sqrt{b}$ for all $a,b\in\mathbb{R}_+$, we have that $\tr(x^{2\alpha-2}Q^\top Q)\leq\tr(x^{2\alpha-1})\tr(Q^\top Qx^{-1})$. Hence, \eqref{eq:condalpha} is implied by $\tr(\Gamma(x)x^{-1}+bx^{-1})\geq 2\tr(x^{2\alpha-1})\tr(Q^\top Qx^{-1})$. Using once again the selfduality gives (d).

Since the trace is the sum of the eigenvalues, $\lambda \geq \lambda^{2\alpha-1}$ for all $\lambda\geq 1$ and $\alpha\in [1/2,1]$ and $\lambda^{2\alpha-1}\leq \lambda +\max_{\lambda\in[0,1]}\left\{\lambda^{2\alpha -1}-\lambda\right\}$ for all $\lambda\in [0,1)$ and $\alpha\in [1/2,1]$, we immediately obtain (e) from (d), because $\max_{\lambda\in[0,1]}\left\{\lambda^{2\alpha -1}-\lambda\right\}=(2\alpha-1)^{2-2\alpha}$. In turn (f) follows from (e) noting that $\max_{\lambda\in[0,1]}\left\{\lambda^{2\alpha -1}-\lambda\right\}\in [0,1]$.

Turning to (g) we have for the right hand side of \eqref{duplo enigmato drifto conditione} in the univariate case
\[
 \ell(x)=2\beta+\Gamma(x)/x+b/x-2Q^2/x^{2-2\alpha}.
\]
Now one notes that the second term is non-negative and that for $b>0$ the term $b/x-2Q^2/x^{2-2\alpha}$ is bounded from below on $\mathbb{R}^+$, because $\lim_{x\to 0,\, x>0}x^{-1}/x^{2\alpha-2}=\infty$. Hence, Theorem  \ref{th: strong solutions} concludes.
\end{proof}
 In the different cases of (ii) a valid choice of $b$ and $\Gamma$ is always obtained by taking them equal to the right hand side of the inequalities. It should be noted that (c) shows that in the positive semidefinite GARCH diffusion generalisation one can always take a linear drift. Likewise, (e) and (f) show that a linear drift is possible for the generalized CIR. For $\alpha=1/2$ the case (d) is again sharp in the univariate setting, but for general dimensions it is a stronger condition than the one given in Corollary \ref{th: strong solutions}.

The last case (g) in particular recovers the well-known univariate result for  $dx_t=(b+ax_t)dt+qx_t^{\alpha}dB_t$ with $b\geq0, q>0, a\in\mathbb{R}$ and $\alpha\in [1/2,1]$. In our matrix-variate case for $\alpha>1/2$ a result similar to the univariate one, viz.   that a strictly positive constant drift is all that is needed to ensure boundary non-attainment, seems to be out of reach. When one tries to use arguments similar to  (e) in general, one would need something like  $
 \tr(bx^{-1})\geq k \tr(x^{2\alpha-1})\tr(Q^\top Qx^{-1})+ K$ with some constants $k>0$ and $ K$ to ensure \eqref{eq:condalpha}. However, when the process comes close to the boundary of the cone, this only means that at least one eigenvalue gets close to zero. Hence,  $\tr(bx^{-1})$ and $\tr(Q^\top Qx^{-1})$ should then go to infinity at a comparable rate. However, all the other eigenvalues of $x$ may still be arbitrarily large and so there is no appropriate upper bound on the term $\tr(x^{2\alpha-1})$.

\section{Proofs}
In this section we gradually prove our main result. As a priori all processes involved are only defined up to a stopping time,
 we collect first some basic definitions regarding stochastic processes defined on stochastic intervals following mainly \cite{maisonneuve1977}.
\begin{definition}
Let $A\in\mathcal{F}$ and  let $T$ be a stopping time.
\begin{itemize}
\item A random variable $X$ on $A$ is a mapping $A\to \mathbb{R}$ which is measurable with respect to the $\sigma$-algebra $A\cap\mathcal{F}$.
\item A family $(X_t)_{t\in \mathbb{R}_+}$ of random variables on $\{t<T\}$ is called a stochastic process on $[0,T)$. If $X_t$ is $\{t<T\}\cap \mathcal{F}_t$-measurable for all $t\in \mathbb{R}_+$, then $X$ is said to be adapted.
\item  An adapted process $M$ on $[0,T)$ is called a continuous local martingale on the interval $[0,T)$ if there exists an increasing sequence of stopping times $(T_n)_{n\in\mathbb{N}}$ and a sequence of continuous  martingales $(M^{(n)})_{n\in\mathbb{N}}$ (in the usual sense on $[0,\infty)$) such that $\lim_{n\to\infty} T_n=T$ a.s. and  $M_{t}=M_t^{(n)}$ on $\{t<T_n\}$.  Other local properties for adapted processes on $[0,T)$ are defined likewise.
\item A semimartingale on $[0,T)$ is the sum of a c\`adl\`ag local martingale on $[0,T)$ and an adapted c\`adl\`ag process of locally finite variation on $[0,T)$.
\item For a continuous local martingale on $[0,T)$ the quadratic variation is the $\mathbb{R}\cup\{\infty\}$-valued stochastic process $[M,M]$ defined by \[[M,M]_t=\sup_{n\in\mathbb{N}}[M^{(n)},M^{(n)}]_{t\wedge T_n}\,\mbox{ for all } t\in\mathbb{R}_+.\]
\end{itemize}
\end{definition}
\subsection{McKean's argument}\label{section: prelim}~\\
In this section we finally establish Proposition \ref{lem: MCKean} which generalises an argument of  \cite[p. 47, Problem 7]{mckean} concerning continuous local martingales on stochastic intervals used, for instance, in \cite{bru_89,bru,norrisrogerswilliams}. We keep the tradition of referring to it as {\it McKean's argument}. Since it may also be helpful in other situations, we state our result and its proof in detail.

\begin{lemma}\label{lem:contlocmart}
 Let $M$ be a continuous local martingale on a stochastic interval $[0,T)$. Then on $\{T>0\}$ it holds almost surely that either $\lim_{t\uparrow T} M_t$ exists in $\mathbb{R}$ or that $\limsup_{t\uparrow T}M_t=-\liminf_{t\uparrow T}M_t=\infty$.
\end{lemma}
\begin{proof}
Combine \cite[Theorem 3.5]{maisonneuve1977} with analogous arguments to the proof of \cite[Chapter V, Proposition 1.8]{revuzyor}.
\end{proof}

\begin{proposition}[McKean's Argument]\label{lem: MCKean}
Let $Z=(Z_s)_{s\in\mathbb{R}_+}$ be an adapted c\`adl\`ag $\mathbb{R}^+\backslash\{0\}$-valued  stochastic process on a
stochastic interval $[0,\tau_0)$ such
that $Z_0>0$ a.s. and $\tau_0=\inf\{0<s\leq\tau_0\,:\, Z_{s-}=0\}$. Suppose
$h:\mathbb R_+\backslash\{0\}\rightarrow \mathbb R$ is continuous and satisfies the following:
\begin{enumerate}
\item For all $t\in[0,\tau_0)$, we have $h(Z_t)=h(Z_0)+M_t+P_t$, where
\begin{enumerate}
\item $P$ is an adapted c\`adl\`ag  process on $[0,\tau_0)$ such that $\inf_{t\in[0,\tau_0\wedge T)}P_t>-\infty$ a.s.
 for each $T\in \mathbb{R}^+\backslash\{0\}$,
\item $M$ is a continuous local martingale on $[0, \tau_0)$ with $M_0=0$,

\end{enumerate}
\item \label{as prop h} and $\lim_{z\downarrow 0}h(z)=-\infty$.
\end{enumerate}
Then $\tau_0=\infty$ a.s.
\end{proposition}
Above, $\tau_0=\inf\{0<s\leq\tau_0\,:\, Z_{s-}=0\}$ is not to be understood as the definition of $\tau_0$, but it means that the already defined stopping time $\tau_0$ is also the first hitting time of $Z_{s-}$ at zero. Since $Z$ is only defined up to time $\tau_0$, one cannot take the infimum over $\mathbb{R}^+$.
\begin{proof}
Since $h(Z_t)_{-}=h(Z_{t-})=h(Z_0)+P_{t-}+M_{t-}$ and $P_{t-}$ is a.s. bounded from below on compacts, we have $\tau_0=\inf\{s>0\,:\, M_{s-}=-\infty\}$ and further $\tau_0>0$ due to the right continuity of $Z$. Assume, by contradiction, that $\tau_0<\infty$ on a set $A\in\mathcal{F}$ with $\mathbb{P}(A)>0$. Hence, $\lim_{t\nearrow \tau_0} M_t=-\infty$ on $A$ and this contradicts Lemma \ref{lem:contlocmart}.
\end{proof}

\subsection{Proof of Theorem 3.4}\label{sec: proof}~\\
Before we provide a proof of  Theorem \ref{th: strong solutions general}, we
recall some elementary identities from matrix calculus and provide
some further technical lemmata. For a differentiable function $f:
M_d\rightarrow \mathbb{R}$, we denote by $\nabla f$ the usual
gradient written in coordinates as
$(\frac{\partial f}{\partial x_{ij}})_{ij}$.

\begin{lemma}\label{lem: matrixcalculus} On $S_d^{++}$, we have
\begin{enumerate}
\item \label{matrix calc 1} $\nabla \det(x)=\det(x)(x^{-1})^\top=\det(x)x^{-1}$,
\item \label{matrix calc 2} $\frac{\partial^2}{\partial x_{ij}\partial
x_{kl}}\det(x)=\det(x)[(x^{-1})_{kl}(x^{-1})_{ij}-(x^{-1})_{il}(x^{-1})_{jk}]$.
\end{enumerate}
\end{lemma}
\begin{proof}

 The first identity in  (i) can be found in \cite[Section 9.10]{MagnusNeudecker1988} and the second is an immediate consequence of restricting to symmetric matrices. Now (ii) follows using  $\frac{\partial}{\partial x_{kl}} x^{-1} =-
x^{-1}\left(\frac{\partial}{\partial x_{kl}}x\right)x^{-1}$ and finally the symmetry:
\begin{align*}
\frac{\partial}{\partial x_{kl}x_{ij}}\det(x)& =  \frac{\partial}{\partial x_{kl}}\left( \det(x) (x^{-1})_{ji}\right) =
\det(x)\left((x^{-1})_{lk}(x^{-1})_{ji} + \frac{\partial}{\partial x_{kl}}
(x^{-1})_{ji}\right)\\
&= \det(x)\left((x^{-1})_{lk}(x^{-1})_{ji} -(x^{-1})_{jk}(x^{-1})_{li}\right).
\end{align*}

\end{proof}

For a semimartingale $X$ we denote by $X^c$ as usual its continuous
part. All semimartingales in the following will have a discontinuous
part of finite variation, i.e. $\sum_{0<s\leq t}\|\Delta X_s\|$ is
finite for all $t\in\mathbb{R}^+$. Thus we define
$X_t^c=X_t-\sum_{0<s\leq t}\Delta X_s$ and note that the quadratic
variation of a semimartingale is the one of its local continuous
martingale part plus the sum of its squared jumps.

The continuous quadratic variation of $X$ solving \eqref{gensde1} is
only influenced by the Brownian terms and, hence, we have a
general version of \cite[Equation (2.4)]{bru} which is proved just as \cite[Lemma 2]{AhdidaAlfonsi2010}:
\begin{lemma}\label{lem: quadr cov1}
Consider the solution $X_t$ of \eqref{gensde1} on $[0, T_x)$. Then
\begin{align*}
\frac{d[X_{ij}, X_{kl}]^c_t}{dt}&=(F F^\top(t,X_{t-}))_{ik}(G^\top
G(t,X_{t-}))_{jl}+(F F^\top(t,X_{t-}))_{il}(G^\top
G(t,X_{t-}))_{jk}\\&+(F F^\top(t,X_{t-}))_{jk}(G^\top
G(t,X_{t-}))_{il}+(F F^\top(t,X_{t-}))_{jl}(G^\top
G(t,X_{t-}))_{ik}.
\end{align*}
Here $G^\top G(t,x):= G(t,x)^\top G(t,x)$ and $FF^\top(t,x):=F(t,x)F(t,x)^\top$ to ease notation.
\end{lemma}
Moreover, we shall need the following result where a Brownian motion on a stochastic interval  $[0,T)$ is defined as a continuous local martingale on $[0,T)$ with $[\beta,\beta]_t=t$.
\begin{lemma}\label{lem: levychar local}
Let $X_t$ be a continuous $S_d^+$-valued adapted c\`adl\`ag
stochastic process on a stochastic interval $[0,T)$ with $T$ being a
predictable stopping time and let $h:M_d\rightarrow M_d$. Then there
exists a one-dimensional Brownian motion $\beta^h$ on $[0,T)$ such
that
\begin{equation}\label{eq: taking brownian traces}
\tr\left(\int_0^t h(X_{u-})dB_u\right)=\int_0^t
\sqrt{\tr(h(X_{u-})^\top h(X_{u-}))}d\beta^h_u
\end{equation}
holds on $[0,T)$.
\end{lemma}
\begin{proof}
We define for $t\in [0,T)$,
\[
\beta_t^h:=\sum_{i,j=1}^d\int_0^t
\frac{h(X_{u-})_{ij}}{\sqrt{\tr(h(X_{u-})^\top
h(X_{u-}))}}dB_{u,ji},
\]
and since the numerator equals zero, whenever the denominator
vanishes, we use the convention $\frac{0}{0}=1$. Clearly for each
$i,j$ and for all $u\in [0,T)$ we have
\[
 \left| \frac{h(X_{u-})_{,ij}}{\sqrt{\tr(h(X_{u-})^\top h(X_{u-}))}} \right|\leq 1
\]
which ensures that $\beta^h$ is well-defined, square-integrable and
a continuous local martingale on $[0,T)$ by stopping at a sequence
of stopping times announcing $T$.
 Furthermore, by construction
\[
[\beta^h,\beta^h]_t=\sum_{i,j=1}^d\int_0^t
\frac{h(X_{u-})_{ij}^2}{\tr(h(X_{u-})^\top h(X_{u-}))}du=t
\]
and therefore  $\beta^h$ is a Brownian motion on $[0,T)$.

Finally by the very definition of $\beta^h$, we have
\[
\tr(h(X_{t-})dB_t)=\sum_{i,j=1}^d
h(X_{t-})_{ij}dB_{t,ji}=\sqrt{\tr(h(X_{t-})^\top
h(X_{t-}))}d\beta^h_t,
\]
which proves identity \eqref{eq: taking brownian traces}.
\end{proof}
Finally, we state a variant of It\^o's formula which we later
employ. It follows easily from the usual versions like \cite[Theorem
3.9.1]{Bichteler2002} by arguments similar to \cite[Theorem
5.4]{maisonneuve1977} and \cite[Proposition 3.4]{barndorffstelzer}.
\begin{lemma}\label{eq: Itooform}
Let $X$ be an $S_d^{++}$-valued semimartingale on a stochastic
interval $[0,T)$ and $f:S_d^{++}\to \mathbb{R}$ a twice continuously
differentiable function. If $X_{t-}\in S_d^{++}$ for all $t\in[0,T)$
and $\sum_{0<s\leq t}\|\Delta X_s\|<\infty$ for $t\in[0,T)$, then
$f(X)$ is a semimartingale on $[0,T)$ and
\begin{align*}
f(X_t)=&f(X_0)+\tr\left(\int_0^t \nabla
f(X_{s-})^\top dX_s^c\right)+\frac{1}{2}\sum_{i,j,k,l=1}^d\int_0^t\frac{\partial^2}{\partial
x_{ij}\partial x_{kl}}f(X_{s-})d[X_{ij},X_{kl}]_s^c\\&+\sum_{0<s\leq
t}\left(f(X_s)-f(X_{s-})\right).
\end{align*}
\end{lemma}

We are now prepared to provide a proof of Theorem \ref{th: strong
solutions general}. Note that to shorten our formulae we use in the
following differential notation and not integral notation as above.

\begin{proof}[Proof of Theorem \ref{th: strong solutions general}] Since \[\operatorname{vec}\left(F(t,X_{t-})dB_t G(t,X_{t-})\right)=\left(G(t,X_{t-})^\top\otimes(F(t,X_{t-})\right)\operatorname{vec}(dB_t),\] it is easy to see that all coefficients of \eqref{gensde1}
are locally Lipschitz and of linear growth. Hence,
standard SDE theory implies again the existence of a unique
c\`adl\`ag adapted non-explosive local strong solution until the first
time $T_x=\inf\{t\geq 0\,:\, X_{t-}\in\partial S_d^+\mbox { or }X_t\not\in S_d^{++}\}$
when $X$ hits the boundary or jumps out of $S_d^{++}$ .  Hence, we  have to show $T_x=\infty$.

By the choice of $K$ and $J$, all jumps have to be positive
semidefinite and hence the solution $X$ cannot jump out of
$S_d^{++}$. This implies that $T_x=\inf\{t\geq 0\,:\,
X_{t-}\in\partial S_d^+\}$.

In the following, all statements are meant to hold on the stochastic
interval $[0,T_x)$. Note that by the right continuity of $X_t$, a.s.
$T_x>0$. Moreover, we set $T_n=\inf\{t\in\mathbb{R}_+\,:\, d(X_t,
\partial S_d^+)\leq 1/n\,\, \textit{or}\,\,\|X_t\|\geq n\}.$ Then
$(T_n)_{n\in\mathbb{N}}$ is an increasing sequence of stopping times
with $\lim_{n\to\infty}T_n=T_x$, hence $T_x$ is predictable.

 We define the following
processes and functions according to the notation of Proposition
\ref{lem: MCKean}:
\begin{equation}\label{eq: trick1}
Z_t:=\det(X_t),\quad h(z):=\ln(z),\quad r_t:=h(Z_t).
\end{equation}
Then $T_x=\inf\{t>0\,:\, r_{t-}=-\infty\}$.

By  Lemma \ref{lem: matrixcalculus} \ref{matrix calc 1} and using
the abbreviation $f=FF^\top$, $g=G^\top G$, we obtain
\begin{align*}
\tr( \nabla(\det(X_{t-}))dX^c_t)=&\det(X_{t-})\biggr[2\sqrt{\tr\left(
f(t,X_{t-})X_{t-}^{-1}g(t,X_{t-}) X^{-1}_{t-}
\right)}dW_t\\&+\tr\left(H(t,X_{t-})X_{t-}^{-1}\right)dt\biggr],
\end{align*}
with some  one-dimensional Brownian motion $W$ on $[0, T_x)$, whose
existence is guaranteed by Lemma \ref{lem: levychar local}.
Furthermore, by Lemma \ref{lem: matrixcalculus} \ref{matrix calc 2},
Lemma \ref{lem: quadr cov1}  and elementary calculations we have
that
\begin{align*}
&\frac{1}{2}\sum_{i,j,k,l}\frac{\partial^2}{\partial x_{{ij}}\partial
x_{{kl}}}\det( X_{t-}) d[X_{ij},X_{kl}]^c_t\\&\quad=\frac{\det(X_{t-})}{2}\sum_{i,j,k,l}\biggl[\bigl((X_{t-}^{-1})_{kl}(X_{t-}^{-1})_{ij}-(X_{t-}^{-1})_{il}(X_{t-}^{-1})_{jk}\bigr)\bigl(f(t,X_{t-})_{ik}g(t,X_{t-})_{jl}\\&\quad\quad+f(t,X_{t-})_{il}g(t,X_{t-})_{jk}+f(t,X_{t-})_{jk}g(t,X_{t-})_{il}+f(t,X_{t-})_{jl}g(t,X_{t-})_{ik}\bigr)\biggr]
\\&\quad=\det
(X_{t-})\left(\tr(f(t,X_{t-})X_{t-}^{-1}g(t,X_{t-}) X^{-1}_{t-})-\tr(f(t,X_{t-})X_{t-}^{-1})\tr(g(t,X_{t-}) X^{-1}_{t-})\right)dt.
\end{align*}
 According
to It\^o's formula, Lemma \ref{eq: Itooform}, we therefore obtain by
summing up the two equations,
\begin{align*}
dZ_t=&2\det(X_{t-})\sqrt{\tr(f(t,X_{t-})X_{t-}^{-1}g(t,X_{t-}) X^{-1}_{t-})}dW_t+\det(X_t)-\det(X_{t-})\\&+ \det(X_{t-})\biggl[
\tr(H(t,X_{t-})X_{t-}^{-1})+\tr(f(t,X_{t-})X_{t-}^{-1}g(t,X_{t-}) X^{-1}_{t-})\\&-\tr(f(t,X_{t-})X_{t-}^{-1})\tr(g(t,X_{t-}) X^{-1}_{t-})\biggr]dt.
\end{align*}
Using again It\^o's formula, we have
\begin{align*}
dr_t=&2\sqrt{\tr(f(t,X_{t-})X_{t-}^{-1}g(t,X_{t-}) X^{-1}_{t-})}dW_t+\ln(\det(X_t))-\ln(\det(X_{t-}))\\&+ \biggl[
\tr(H(t,X_{t-})X_{t-}^{-1})-\tr(f(t,X_{t-})X_{t-}^{-1}g(t,X_{t-}) X^{-1}_{t-})\\&-\tr(f(t,X_{t-})X_{t-}^{-1})\tr(g(t,X_{t-}) X^{-1}_{t-})\biggr]dt.
\end{align*}
Hence, we have $r_t=r_0+M_t+P_t$, where
\begin{align*}
M_t=&2\int_0^t\sqrt{\tr(f(s,X_{s-})X_{s-}^{-1}g(s,X_{s-}) X^{-1}_{s-})}dW_s,\\
P_t=&\int_0^t\biggl[
\tr(H(s,X_{s-})X_{s-}^{-1})-\tr(f(s,X_{s-})X_{s-}^{-1}g(s,X_{s-}) X^{-1}_{s-})\\&-\tr(f(s,X_{s-})X_{s-}^{-1})\tr(g(s,X_{s-}) X^{-1}_{s-})\biggr]ds +\sum_{0<s\leq t}\left(\ln(\det(X_s))-\ln(\det(X_{s-}))\right).
\end{align*}
 We infer that $(M^{(n)}_t)_{t\geq
0}$ defined by $$M^{(n)}_t:=2\int_0^t\sqrt{\tr(f(s,X^{T_n}_{s-})(X^{T_n}_{s-})^{-1}g(s,X^{T_n}_{s-}) (X^{T_n}_{s-})^{-1})}dW_s$$ is a continuous martingale. Obviously,
$M_t=M^{(n)}_t$ on $\{t<T_n\}$ and thus $M$ is a continuous local
martingale on $[0,T_x)$.  Furthermore, $X_s-X_{s-}\succeq 0$
for all $s\in [0,T)$ and hence $\det(X_s)\geq \det(X_{s-}) $ using
\cite[Corollary 4.3.3]{Hornetal1990}. Therefore, we have that $P_t\geq \int_0^tc(s)ds$ on $[0,T_x)$.

Finally, by Proposition \ref{lem: MCKean} we have that $T_x=\infty$
a.s. noting that $c$ is assumed to be locally integrable.
\end{proof}
\begin{remark}\label{remark on bru}
Bru's method for proving her proposition \ref{th: bru} for Wishart diffusions   consists of the following two steps{:}
\begin{enumerate}
\item First assume $\beta=0$. By applying the original McKean's argument twice, one
derives that $h(\det(X))$ is a local martingale. This is proved separately for $\delta=d+1$ and $\delta>d+1$  by choosing $h(z)=\ln(z)$ in the first case and $h(z)=z^{d+1-\delta}$ in the second one. Therefore, the existence of a unique global strong solution on $S_d^{++}$ is settled.
\item \label{step2} One may therefore suppose that $X_t$ is an $S_d^{++}$-valued solution on $[0,\infty)$ of
\[
dX_t=\sqrt{X_t}dB_tQ+Q^\top dB_t^{\top}\sqrt{X_t}+\delta Q^\top Q dt,\quad X_0=x\in S_d^{++}.
\]
where $Q\in GL(d)$ and $\delta\geq d+1$. Now, Girsanov's Theorem is applied which allows to introduce a drift
by changing to an equivalent probability measure. This step generalises a one-dimensional method by Pitman and Yor, see
\cite[p. 748]{bru}. The involved arguments and calculations, which are not presented in detail in \cite{bru}, appear rather complicated and work seemingly  only in the special case given in Proposition \ref{th: bru} {\ref{stat: bru1}}, \ref{stat: bru2}.
\end{enumerate}
The technical details of \cite{bru} concerning strong solutions are explained in more detail in \cite{Pfaffl}.
\end{remark}
Our proof above circumvented the problems associated to the use of Girsanov's theorem by extending the approach outlined in (i).
\section{Conclusion}\label{sec: final}
In this paper we have extended the previously known  sufficient
boundary non-attain\-ment conditions for certain Wishart processes to  more general SDEs on $S_d^{++}$, which include  affine diffusions
with state-independent jumps of finite variation. This allowed to infer the existence of strong
solutions of a large class of affine matrix valued processes. Moreover, we have thus obtained strong existence results for SDEs which can be considered as positive semidefinite extensions of GARCH diffusions and generalised Cox-Ingersoll-Ross processes.

However, this results in several open questions related to our SDE \eqref{eq: affine process} which will hopefully  be addressed in future work.  The following questions are beyond the scope of the present paper,
since they are obviously rather non-trivial and apparently need very different techniques than the ones employed here.
For $d>1$ and the Wishart diffusions it is not clear, whether the condition
$b\succeq(d+1)Q^\top Q$ for the drift is a necessary non-attainability
condition or not. Only in the case $\beta=0,\Gamma=0, Q=I_d$ and $b=\delta I_d$ with $\delta\in(d-1,d+1)$ it is known from \cite[Theorem 1.4]{donatimartin} that the boundary is hit. On the other hand, one  knows that in the case $d=1$
pathwise uniqueness holds, hence there exists a strong solution for all
$b\succeq 0$ (even in the general setting of CBI processes, see \cite[Theorem 5.1]{dawsonli}).  For $d\geq 2$, the situation seems in general to be rather complicated and therefore existence of global strong solutions remains an open problem  when $b\nsucceq (d+1)Q^\top Q$ (and the conditions for the existence of weak solutions of \cite{CFMT} are satisfied). Likewise, it is a very interesting problem in the case of the GCIR processes with $\alpha>1/2$ whether a state dependent drift away from the boundary is really necessary and what happens if one has only a constant drift towards the interior of $S_d^+$.

Finally, we remark that our method of proof could be generalised to state-spaces $D$ other than $S_d^+$,
as long as the existence of an appropriate function $g: D\rightarrow \mathbb R_+$ is guaranteed,
such that $g^{-1}(0)=\partial D$. For instance, similar (but simpler) arguments to the ones of the proof of Theorem \ref{th: strong solutions general} yield a rigorous proof of the non-attainment condition formulated in \cite[Section 6]{cheridito}
for affine jump diffusions on the canonical state space $\mathbb R_+^m\times \mathbb R^n$. Here one takes $g(x_1,x_2,\ldots,x_m)=x_1\cdot x_2\cdot\dots\cdot x_m$.

\bibliographystyle{abbrvnat}

\end{document}